\documentclass{amsart}
\usepackage{amssymb, amsmath}
\usepackage[latin1]{inputenc}

\newtheorem{theorem}{Theorem}
\newtheorem{lemma}[theorem]{Lemma}
\newtheorem{remark}[theorem]{Remark}
\newtheorem{definition}[theorem]{Definition}

\newtheorem{example}[theorem]{Example}

\DeclareMathOperator{\Hes}{\rm Hes}
\DeclareMathOperator{\hes}{\rm hes}
\DeclareMathOperator{\Ric}{\rm Ric}

\DeclareMathOperator{\Id}{\rm Id}

\DeclareMathOperator{\coef}{\rm Coef}

\begin{document}

\title[Half conformally flat generalized quasi-Einstein manifolds]
{Half conformally flat\\ generalized quasi-Einstein manifolds}
\author{M. Brozos-V\'{a}zquez \, E. Garc\'{i}a-R\'{i}o\, P. Gilkey \, X. Valle-Regueiro}
\address{MBV: Universidade da Coru\~na, Differential Geometry and its Applications Research Group, Escola Polit\'ecnica Superior, 15403 Ferrol,  Spain}
\email{miguel.brozos.vazquez@udc.gal}
\address{EGR-XVR: Faculty of Mathematics,
University of Santiago de Compostela,
15782 Santiago de Compostela, Spain}
\email{eduardo.garcia.rio@usc.es; javier.valle@usc.es}
\address{PG: Mathematics Department, \; University of Oregon, \;\;
Eugene \; OR 97403, \; USA}
\email{gilkey@uoregon.edu}
\thanks{Supported by projects EM2014/009, MTM2013-41335-P with FEDER funds, and MTM2016-75897-P (Spain).}

\begin{abstract}
We provide classification results for and examples of half conformally flat generalized quasi Einstein manifolds of signature $(2,2)$. 
This analysis leads to a natural equation in affine geometry called the affine quasi-Einstein equation that we explore in further detail. 
\end{abstract}

\maketitle

\section{Introduction}

The analytical study of differential equations often focuses on the existence and uniqueness of solutions on a given domain.
From a geometric point of view, the converse question is also of interest. Given a differential equation, one may look for a 
manifold that supports a non-trivial solution and one might ask about the local/global geometry of the manifold, thus leading 
to an analytical characterization of a manifold structure by a differential equation if this manifold corresponds to a unique domain 
where the given equation has a non-trivial solution. Clearly one may not expect any positive answer for arbitrary equations but 
there are important examples when the equation has some geometrical/physical meaning. The equation of Obata \cite{obata} is a
typical example; see also the discussion in \cite{david-pontecorvo,obata2,ranjan}.

Let $\rho$ be the Ricci tensor of a pseudo-Riemannian manifold $\mathcal{M}=(M,g)$. If $f$ is a smooth function on $M$, let $\Hes_f$ be the
Hessian. Both $\rho$ and $\Hes_f$ are $(0,2)$-tensor fields on $M$; we refer to Section~\ref{S1.5} for a precise definition. The
generalized quasi-Einstein equation links these two objects with the metric tensor in a very natural
fashion. This single equation extends equations studied previously such as the equation of Obata \cite{obata}, 
the M\"obius equation \cite{xu}, the Einstein equation, and the gradient Ricci soliton equation as we shall see
in the discussion given below. In this paper, we examine the generalized quasi-Einstein equation (see Equation~(\ref{eq:general quasi-Einstein}) below)
in the setting of half conformally flat manifolds of signature $(2,2)$.

\begin{definition} \rm A quadruple $(M,g,f,\mu)$, where $(M,g)$ is a pseudo-Riemannian manifold of dimension $n$, $f$ is a smooth function on $M$,
and $\mu\in \mathbb{R}$, is said to be a \emph{generalized quasi-Einstein manifold} if the tensor ${\Hes}_f+\rho-\mu df\otimes df$
is a multiple of the metric, i.e. if the following equation (which is called the generalized quasi-Einstein equation) is satisfied:
\begin{equation}\label{eq:general quasi-Einstein}
{\Hes}_f+\rho-\mu df\otimes df=\lambda\, g\text{ for some }\lambda\in \mathcal{C}^\infty(M)\,.
\end{equation}\end{definition}

There are several interesting families of generalized quasi-Einstein manifolds that have been considered in the literature previously:

\begin{example}[\bf Einstein manifolds]\rm One can recover the Einstein equation by letting $f$ be constant
in Equation~\eqref{eq:general quasi-Einstein}. Consequently any Einstein manifold is in fact a generalized quasi-Einstein manifold. 
Suppose on the other hand that $\mathcal{M}$ is Einstein. We consider Equation~\eqref{eq:general quasi-Einstein} for $\mu\neq 0$.
The change of variable $h=e^{\mu f}$ provides the equivalent equation $\frac{1}{\mu h}\Hes_{h}+\rho=\lambda g$. 
Let $\tau$ be the scalar curvature; as $\mathcal{M}$ is Einstein, $\rho=\frac{\tau}{n} g$. Multiplying by $\mu h$ converts
the relation $\frac{1}{\mu h}\Hes_{h}+\rho=\lambda g$ into the equation
\begin{equation}\label{eq:Moebius}
\textstyle\Hes_{h}+\mu h(\frac{\tau}{n}-\lambda)g=0\,.
\end{equation}
This is precisely the Equation of M\"obius, where $\Delta h=\mu (\tau-n \lambda) h$ (see, for example, \cite{xu}). 
Moreover, if $\lambda$ is constant, then Equation~\eqref{eq:Moebius} resembles the equation of Obata $\Hes_{h}+\kappa h g=0$
since $\kappa=\mu(\frac{\tau}{n}-\lambda)$ is a constant (see \cite{obata}).\end{example}

\begin{example}[\bf Gradient Ricci almost solitons]\rm\label{ex:3} For $\mu=0$, Equation~\eqref{eq:general quasi-Einstein}
corresponds to the gradient Ricci almost soliton equation (see, for example, \cite{BBR--2012,Brozos-Garcia-Valle,PRRS}).
In particular, if $\lambda$ is constant, then one obtains the gradient Ricci soliton equation (see 
\cite{cao-chen1,cao-wang-zhang,chen-wang,munteanu-sesum} and references therein), which identifies 
self-similar solutions of the Ricci flow: $\frac{\partial}{\partial t}g(t)=-2\rho(t)$. Although gradient Ricci solitons are a 
special case of quasi-Einstein metrics, they exhibit quite different properties (see \cite{Case1}).
We emphasize that the gradient Ricci almost soliton equation is not just a formal generalization of the Ricci soliton equation, 
but includes families of self-similar solutions of other geometric flows such as the \emph{Ricci-Bourguignon flow} \cite{CCDMM}.
This flow is defined for a $\kappa\in\mathbb{R}$  by the evolution equation
$\partial_tg(t)=-2(\rho(t)-\kappa\tau(t)\, g(t))$. The self-similar solutions of this flow are gradient Ricci almost solitons with soliton function 
$\lambda=\kappa\,\tau+\nu$ (for some $\nu\in\mathbb{R}$) and are called \emph{$\kappa$-Einstein solitons}
(see  \cite{CMMR, CaMa} for further details).\end{example}

\begin{example}[\bf Conformally Einstein manifolds]\label{subsubsect:conf-Einstein}\rm For $n\geq 3$, $(M,g, f, -\frac{1}{n-2})$ is a generalized quasi-Einstein manifold if and only if $(M,e^{-\frac{2}{n-2}f}g)$ is Einstein. Consequently, the parameter $\mu=-\frac{1}{n-2}$ is a distinguished value which is often exceptional, see Theorem~\ref{th:non-isotropic}, Theorem~\ref{th:isotropic} and 
Example~\ref{remark:conf-Einstein} for example. We refer to \cite{brinkmann, gover-nurowski} for more detailed information on conformally Einstein manifolds.\end{example}
 
\begin{example}[\bf Static space-times] \rm For $\mu=1$, the change of variable $h=e^{-f}$ transforms Equation~\eqref{eq:general quasi-Einstein} 
into the equation $\Hes_h- h \rho=- h \lambda g$. If $\lambda=-\frac{\Delta h}{h}$, then this equation becomes
$\Hes_h-h\rho=\Delta h g$. This is the defining equation of the so-called \emph{static manifolds}
 that arise in the study of static space-times (we refer to \cite{Kobayashi,Kobayashi-obata} for further details).
 \end{example}

\begin{example}[\bf Quasi-Einstein manifolds and Einstein warped products]\rm A solution  Equation~\eqref{eq:general quasi-Einstein}
with $\lambda$ constant is said to be {\it quasi-Einstein} and the resulting equation is called the {\it quasi-Einstein equation}. 
Let $B\times_\varphi F$ be an Einstein warped product with
$\operatorname{dim} F=r$ and warping function $\varphi=e^{-\frac{f}{r}}$. The structure on the base  $\left(B, g, f, \frac{1}{r}\right)$ is  then quasi-Einstein.
Conversely, starting with a quasi-Einstein manifold $\left(B, g, f, \mu\right)$ where $\mu=\frac{1}{r}$ for $r$ a positive integer, 
there exist appropriate Einstein fibers $F$ so that $B\times_\varphi F$ is Einstein, see, for example, \cite{KimKim}.\end{example}

\begin{remark}\rm Even more general classes exist in the literature \cite{Catino12,chen-liang-zhu,GJS,neto}.\end{remark}

\subsection{Motivation}
Equation~\eqref{eq:general quasi-Einstein} provides information on the curvature of the manifold since it involves the associated Ricci tensor. 
We shall impose various conditions on the Weyl tensor to obtain related families of generalized quasi-Einstein manifolds. 
One could assume, for example, that $\mathcal{M}$ is locally conformally flat; this condition turns out to be quite restrictive. We refer, for example,
to the discussion in  \cite{cao-chen1,cao-wang-zhang} 
in relation to gradient Ricci solitons and to the discussion in \cite{LCFLorentzianQE,Catino13} for quasi-Einstein manifolds. Other weaker conditions
were considered in \cite{Catino12} for a slightly more general class of manifolds than the one we consider here. Suppose that $\mu$ is not
assumed to be constant. It is known that  $4$-dimensional generalized quasi-Einstein manifolds with harmonic
Weyl tensor and zero radial Weyl curvature are indeed locally conformally flat in Riemannian signature.
Associated rigidity results are available. See, for example,
\cite{BR-2014,Case1,Hu-Li-Zhai} and the references therein.

Other natural conditions on the conformal curvature were previously considered for $4$-dimensional manifolds and particular families of generalized
quasi-Einstein manifolds.  One says that $\mathcal{M}=(M,g)$ is {\it half conformally flat} if $\mathcal{M}$ is either self-dual or anti-self-dual.
The notation is chosen to avoid
specifying the orientation. One has that half conformally flat quasi-Einstein manifolds are locally conformally flat in the
Riemannian setting \cite{CatinoDGA,deng}. We refer to \cite{chen-wang} for the gradient Ricci soliton case and to \cite{neto} for related work.

The key point in this analysis is that, in definite signature, the level hypersurfaces of the potential function are non-degenerate and have constant sectional
curvature. However, this need no longer hold true if the signature is indefinite. In this setting, the metric may be degenerate on the level hypersurfaces of the
potential function. This gives rise to null parallel distributions (Walker structures) and to examples which are not locally
conformally flat (see \cite{MeE,Brozos-Garcia-Valle}).

In this paper, we shall examine $4$-dimensional generalized quasi-Einstein manifolds in neutral signature $(2,2)$.
We wish to find examples which are half conformally flat, but not locally conformally flat.  The analysis depends to a large extent
on the nature of the vector field $\nabla f$. If $\|\nabla f\|\neq 0$, then $\mathcal{M}$ is said to be {\it non-isotropic} while if $\|\nabla f\|=0$ but $\nabla f\ne0$,
then $\mathcal{M}$ is said to be {\it isotropic}. We shall see that solutions of Equation~\eqref{eq:general quasi-Einstein} in the non-isotropic setting
behave very much like solutions of Equation~\eqref{eq:general quasi-Einstein} in  Riemannian signature. The isotropic setting has genuinely new
phenomena not present in the Riemannian setting and Walker structures play a fundamental role. We are interested in the local theory and can
restrict to an arbitrarily small open neighborhood $\mathcal{O}$ of the point $P$ of $M$ in question. We shall assume $\nabla f$ does not vanish
on $\mathcal{O}$. We shall also assume either $\|\nabla f\|$ never vanishes on $\mathcal{O}$
or that $\|\nabla f\|$ vanishes identically on $\mathcal{O}$. We shall not treat the mixed case
where the type of $\nabla f$ changes.

\subsection{Walker manifolds}\label{sect:Walker}
We now summarize the basic facts we shall need about Walker geometry and introduce some important families of Walker manifolds.
Following the seminal work of Walker \cite{Walker} (see also \cite{DeR}), a pseudo-Riemannian manifold $\mathcal{M}=(M,g_{\mathcal{W}})$ is said to
be a \emph{Walker manifold} if $\mathcal{M}$ admits a null parallel distribution $\mathfrak{D}$. We shall work in signature $(2,2)$ and
assume that $\mathfrak{D}$ is 2-dimensional. There are then the canonical local coordinates $(x^1,x^2,x_{1'},x_{2'})$ of Walker. 
To simplify the notation, let $\partial_{x^i}:=\frac{\partial}{\partial x^i}$ and $\partial_{x_{i'}}:=\frac{\partial}{\partial x_{i'}}$ for $i=1,2$. 
Let $\circ$ denote the symmetric product. We adopt the {\it Einstein convention} and sum over repeated indices. There are smooth functions
$a_{ij}=a_{ij}(x^1,x^2,x_{1^\prime},x_{2^\prime})$ which are defined locally on $M$ so that the metric
 $g_\mathcal{W}$ and the distribution $\mathfrak{D}$ take the form
\begin{equation}\label{eq:walker-coordinates}
g_\mathcal{W}= 2 dx^i\circ dx_{i'}+ a_{ij} dx^i\circ dx^j\text{ and }\mathfrak{D}=\operatorname{span}\{\partial_{x_{1'}},\partial_{ x_{2'}}\}\,.
\end{equation}
Any Walker manifold has a canonical orientation \cite{derd-book,derd} which is linked to the orientation of the null distribution $\mathfrak{D}$.
If $\star$ is the Hodge operator, we require that $\star\mathfrak{D}=\mathfrak{D}$ so $\mathfrak{D}$ is self-dual or, equivalently, 
$\star(d{x_{1'}}\wedge d{x_{2'}})=d{x_{1'}}\wedge d{x_{2'}}$. We fix this orientation henceforth.

Let $D$ be a torsion free connection on a surface $\Sigma$. If $(x^1,x^2)$ are local coordinates on $\Sigma$, let $(x_{1'},x_{2'})$ be the corresponding
dual coordinates on the cotangent bundle $T^*\Sigma$; if $\omega$ is a 1-form, we can express $\omega=x_{1'}dx^1+x_{2'}dx^2$. Let $\Phi$ be an 
auxiliary symmetric $(0,2)$-tensor field. Let $\Gamma_{ij}{}^k$ be the Christoffel symbols of the connection $D$.
The {\it deformed Riemannian extension} is defined by setting:
\begin{equation}\label{Eq4}
g_{D,\Phi}= 2 dx^i \circ dx_{i'} +\left\{-2x_{k'} {}^D\Gamma_{ij}{}^k+\Phi_{ij}\right\}dx^i \circ dx^j\,.
\end{equation}
This is an invariantly defined neutral signature metric on the cotangent bundle. Deformed Riemannian extensions were used in \cite{MeE} to describe
self-dual gradient Ricci solitons which are not locally conformally flat. 
More generally, let $T=(T_i^j)$ and $S=(S_i^j)$ be endomorphisms of the tangent
bundle of $\Sigma$. The {\it modified Riemannian extension} is defined \cite{CLGRGVL} by setting:
\begin{equation}\label{eq:mre}
g_{D,\Phi,T,S}= 2 dx^i \circ dx_{i'} +\left\{\frac{1}{2} x_{r'} x_{s'}(T^r_i S^s_j +T^r_j S^s_i)-2x_{k'} {}^D\Gamma_{ij}^k+\Phi_{ij}\right\}dx^i \circ dx^j\,.
\end{equation}
Modified Riemannian extensions were used in \cite{CLGRGVL} to describe Walker manifolds which are self-dual.
This metric and other related metrics appear in many contexts; see, for example, \cite{afifi, MeE,walker-metrics,CLGRGVL}. 

Although it is possible to show directly that the metrics of Equation~(\ref{Eq4}) and of Equation~(\ref{eq:mre}) are defined invariantly,
it is worth introducing a coordinate free formalism as we
shall need the requisite notation subsequently in any event. Let $\pi:T^*\Sigma\rightarrow\Sigma$ be the natural projection.
The geometries of the affine surface $\Sigma$ and of the cotangent bundle $T^\ast \Sigma$ are linked through
\emph{evaluation maps} and \emph{complete lifts}. Given a vector field $X$ on $\Sigma$,  the \emph{evaluation map} 
$\iota X$ is the  function on $T^*\Sigma$ which is characterized by the identity
$$(\iota X)(p,\omega):=\omega_p(X(p))\text{ for }(p,\omega)\in T^\ast \Sigma\,.
$$
Vector fields on $T^*\Sigma$ are determined 
by their action on evaluation maps (see \cite{YI}). For a vector field $X$ on $\Sigma$, the \emph{complete lift} of $X$, which is denoted by $X^C$,
is the vector field on $T^*\Sigma$ that satisfies $X^C(\iota Z)=\iota [X,Z]$ for any vector field $Z$ on $\Sigma$.
The deformed Riemannian extension
of Equation~(\ref{Eq4}) is characterized invariantly by its action on complete lifts:
\[ 
g_{D,\Phi}(X^C,Y^C)=-\iota(D_XY+D_YX)+\Phi(X,Y)\,.
\]
Similarly, if $T=(T_i^j)$ is an endomorphism of $T\Sigma$, then
the evaluation $\iota T$ is a $1$-form on $T^*\Sigma$ which is characterized by the property $(\iota T)(X^C)=\iota (T(X))$.
The  metric of Equation~(\ref{eq:mre}) is given invariantly by the equation:
\begin{equation}\label{eq:metrictensorofriemannextensions}
g_{D,\Phi,T,S} = \iota T \circ \iota S +g_{D,\Phi}\,.
\end{equation}

\subsection{Main results}
Let $\mathcal{M}$ be a 4-dimensional half conformally flat generalized quasi-Einstein manifold. If $\mathcal{M}$ is Riemannian,
under fairly mild assumptions, one can show that $\mathcal{M}$ is locally
conformally flat; see, for example, the discussion in \cite{CatinoDGA,Catino12,chen-wang}. By contrast, in the signature $(2,2)$ setting, there are examples which are half
conformally flat, but not locally conformally flat (see Remark~\ref{R11} below). We work purely locally and shall replace the
original manifold by an arbitrarily small neighborhood of the point in question.
As noted above, we shall either assume that $\|\nabla f\|\ne0$ or that $\|\nabla f\|$ vanishes identically but $\nabla f\ne0$; we
shall not consider the ``mixed" case since we are especially interested in describing self-dual generalized quasi-Einstein metrics that are not locally conformally flat. We shall establish the following results in Section~\ref{S2} and in Section~\ref{S3}, respectively.

\begin{theorem}\label{th:non-isotropic}
Let $(M,g,f,\mu)$ be a half conformally flat generalized quasi-Einstein manifold of signature $(2,2)$ with $\mu\neq -\frac{1}{2}$ and 
$\|\nabla f\|\neq 0$. Then $(M,g)$ is conformally flat and is locally isometric to a warped product of the form $I\times_\varphi N$, where
$I\subset \mathbb{R}$ and $N$ is of constant sectional curvature.
\end{theorem}

\begin{theorem}\label{th:walker}
Let $(M,g,f,\mu)$ be a half conformally flat generalized quasi-Einstein manifold of signature $(2,2)$ with $\mu\neq -\frac{1}{2}$, with $\nabla f\ne0$,
and with $\|\nabla f\|= 0$. Then $(M,g)$ is a Walker manifold with a $2$-dimensional null parallel distribution so the metric $g$ has the form
of Equation~(\ref{eq:walker-coordinates}) in some suitable system of local coordinates.
\end{theorem}

These two results are not sensitive to the choice of the orientation. However, as noted above, the Walker manifolds we shall be considering
come equipped with natural orientations.  Adopt the notation of Equation~(\ref{Eq4}) and of Equation~(\ref{eq:mre}).
We will establish the following result in Section~\ref{S4}.

\begin{theorem}\label{Thm10} Let $(\Sigma,D)$ be an affine surface, let $\hat f\in \mathcal{C}^\infty(\Sigma)$ and let $f=\pi^*\hat f$.
\begin{enumerate}
\item Let $\Phi$ be arbitrary. Suppose $\hat f$ satisfies
\begin{equation}\label{eq:QEE-affine}
\operatorname{Hes}^D_{\hat f}+2\rho^D_s-\mu d\hat f\otimes d\hat f=0\text{ for some }\mu\in\mathbb{R}\,.
\end{equation}
Then $(T^*\Sigma,g_{D,\Phi},f,\mu)$ is a
self-dual isotropic quasi-Einstein Walker manifold with $\lambda=0$.
\item Let $\Phi=\frac{ 2}{C}e^{\hat f}(\operatorname{Hes}^D_{\hat f}+2\rho^D_s-\mu d\hat f\otimes d\hat f)$
and let $T=C e^{-\hat f}\Id$ for $\mu\in \mathbb{R}$ and for $0\ne C\in \mathbb{R}$. 
Then $(T^*\Sigma,g_{D,\Phi,T,\operatorname{Id}},f,\mu)$ is a
self-dual isotropic generalized quasi-Einstein Walker manifold with $\lambda=\frac32 C e^{-f}$.
\end{enumerate}\end{theorem}

\begin{remark}\label{R11}\rm
The manifolds described in Assertion~(1) are not locally conformally flat in general. If $(\Sigma,D)$ is not projectively flat (see Definition~\ref{D21}),
then the deformed Riemannian extension  $(T^*\Sigma,g_{D,\Phi})$ is not locally conformally flat. However even if $(\Sigma,D)$ is projectively flat, 
one can choose
$\Phi$ so that $(T^*\Sigma,g_{D,\Phi})$ is not locally conformally flat. The metrics of Assertion~(2) are self-dual but never anti-self-dual.
And if $\hat f$ is non-constant, then $\lambda$ is non-constant so $(T^*\Sigma,g_{D,\Phi,T,\operatorname{Id}},f,\mu)$ is not quasi-Einstein.
Anti-self-dual modified Riemannian extensions  $(T^*\Sigma,g_{D,\Phi,T,\operatorname{Id}})$ have zero scalar curvature. This does not happen for
these examples since $\lambda=\tau/4$, as we will see in Section~\ref{S4}. Moreover, notice that these manifolds are $\frac14$-Einstein solitons if $\mu=0$ (see Example~\ref{ex:3}).
\end{remark}

The following result is a partial converse to Theorem~\ref{Thm10} and describes the possible local forms of
self-dual isotropic generalized quasi-Einstein metrics. 

\begin{theorem}\label{th:isotropic}
Let $(M,g,f,\mu)$ be a self-dual generalized quasi-Einstein  manifold of signature $(2,2)$ with $\mu\neq -\frac{1}{2}$ and $\|\nabla f\|=0$
which is not Ricci flat.  
\begin{enumerate}
\item If $\lambda$ is constant, then $\lambda=0$ and $(M,g,f,\mu)$ is locally isometric to a manifold which has the form given in
Assertion~(1) of Theorem~\ref{Thm10}.
\item if $\lambda$ is non-constant, then $(M,g,f,\mu)$ is locally isometric to a manifold which has the form given in Assertion~(2) of
Theorem~\ref{Thm10}.
\end{enumerate}
\end{theorem}

\subsection{Outline of the paper} 
In Section~\ref{S2}, we provide some general results concerning generalized quasi-Einstein manifolds.
In Section~\ref{S3}, we examine
the non-isotropic setting and establish Theorem~\ref{th:non-isotropic}. In Section~\ref{S4}, we examine
the isotropic setting and establish Theorem~\ref{th:walker}. We continue our analysis of the isotropic setting
in Section~\ref{S5} and establish Theorem~\ref{Thm10}, and Theorem~\ref{th:isotropic}. 
Let $(M,D)$ be an affine manifold of arbitrary dimension. 
In Section~\ref{S6}, we linearize Equation~(\ref{eq:QEE-affine}) to define an equivalent equation called the {\it affine quasi-Einstein equation} 
\begin{equation}\label{x25}
\operatorname{Hes}^D_{\hat h}=\mu\hat h\rho^D_s\text{ for }\hat h\in \mathcal{C}^\infty(M)\,.
\end{equation}
We relate Equation \ref{x25} to the geometries described above and give some of its basic properties. This equation is essentially the only linear second order
partial differential equation which is natural in the context of affine geometry and is important in its own right.
The remaining part of the paper deals with examples that illustrate important aspects of the equation. In Section~\ref{S7} we give solutions to Equation~(\ref{x25}) which are based on homogeneous
affine surface geometries; this gives rise to purely algebraic considerations. In Section~\ref{S8}, we use the Cauchy-Kovalevskaya
Theorem to construct inhomogeneous surface geometries solving Equation~(\ref{x25}).
In Section~\ref{S9}, we give examples which are anti-self-dual but 
not self-dual and consequently do not fit into the hypothesis of Theorem~\ref{th:isotropic}.

\subsection{Notational conventions}\label{S1.5}
Let $\nabla$ be the Levi-Civita connection of a pseudo-Riemannian manifold $\mathcal{M}=(M,g)$ of dimension $n$. Let
$\mathcal{R}$ be the curvature operator, $R$ be the curvature tensor, $\rho$ be the Ricci tensor,
Ric be the Ricci operator, $\tau$ be the scalar curvature, $W$ be the Weyl tensor, and $C$ be the Cotton tensor:
\medbreak\quad
$\mathcal{R}(X,Y)Z=\nabla_{[X,Y]}Z-[\nabla_X,\nabla_Y]Z,\quad R(X,Y,Z,T)=g(\mathcal{R}(X,Y)Z,T)$,
\medbreak\quad
$\rho(X,Y)=\operatorname{Tr}\{Z\to \mathcal{R}(X,Z)Y\}$,\quad $g(\Ric(X),Y)=\rho(X,Y)$,\quad
$\tau=\operatorname{Tr}(\Ric)$,
\medbreak\quad
$W(X,Y,Z,T)=R(X,Y,Z,T)
+\frac{\tau}{(n-1)(n-2)}\{g(X,Z)g(Y,T)-g(X,T)g(Y,Z))\}$
\medbreak\qquad$+\frac{1}{(n-2)}\{\rho(X,T)g(Y,Z)-\rho(X,Z)g(Y,T)+\rho(Y,Z)g(X,T)-\rho(Y,T)g(X,Z)\}$,
\medbreak\quad$C(X,Y,Z)=-\frac{n-2}{n-3}\operatorname{div}_4 W(X,Y,Z)$.
\medbreak\noindent The \emph{Hessian tensor} $\operatorname{Hes}_f(X,Y)$, \emph{Hessian operator} $\hes_f$, and 
the \emph{Laplacian} $\Delta f$ of a
smooth function $f$ are given by:
\medbreak\quad $\operatorname{Hes}_f(X,Y)=(\nabla_X df)(Y)=XY(f)-(\nabla_X Y)(f)$,
\medbreak\quad
$g(\hes_f(X),Y)=\Hes_f(X,Y)$,\quad$\Delta f=\operatorname{Tr}(\hes_f)$.
\medbreak\noindent Note that $\mathcal{R}$, $\rho$, and $\operatorname{Hes}_f$ are well defined in the context of affine geometry; the other
tensors and operators are not. Since we are assuming the connection $D$ is torsion free, the Hessian is symmetric. However,
even with this assumption, the Ricci tensor need not be symmetric. Consequently, we decompose $\rho^D=\rho^D_s+\rho_a^D$ into the
symmetric and alternating Ricci tensors where 
$$
\rho^D_s(X,Y):=\textstyle\frac12\{\rho^D(X,Y)+\rho^D(Y,X)\}\text{ and }\rho_a^D(X,Y):=\frac12\{\rho^D(X,Y)-\rho^D(Y,X)\}\,.
$$

\section{Generalized quasi-Einstein manifolds}\label{S2}

We now establish some general results concerning generalized quasi-Einstein manifolds that we will use subsequently. 
\begin{lemma}\label{lemma:formulae}
Let $(M,g,f,\mu)$ be a generalized quasi-Einstein manifold. Then
\begin{enumerate}
\item $\tau+\Delta f-\mu \|\nabla f\|^2  = n\lambda$.
\smallbreak\item $\nabla\tau+ \nabla \Delta f-2\mu\hes_f(\nabla f)  = n \nabla\lambda$.
\smallbreak\item  $\nabla \tau+2\mu(\lambda(n-1)-\tau) \nabla f+ 2 (\mu-1)\Ric(\nabla f)=2(n-1)\nabla \lambda$.
\smallbreak\item  $R(X,Y,Z, \nabla f)
=d\lambda(X)g(Y,Z)- d\lambda(Y)g(X,Z)+(\nabla_Y \rho)(X,Z)$\smallbreak\qquad\qquad
$+\mu\left\{df(Y)\Hes_f(X,Z)-df(X)\Hes_f(Y,Z)\right\}$.
\smallbreak\item Let $\eta=\mu(n-2)+1$. Then 
\smallbreak\quad $ W(X,Y,Z,\nabla f)=-C(X,Y,Z)+\frac{\tau\eta\{ df(Y)g(X,Z)-df(X)g(Y,Z)\}}{(n-1)(n-2)}$
\smallbreak\qquad$+\frac{\eta\{\rho(X,\nabla f)g(Y,Z)-\rho(Y,\nabla f)g(X,Z)\}}{(n-1)(n-2)}
+\frac{\eta\{\rho(Y,Z)df(X)-\rho(X,Z)df(Y)\}}{(n-2)}$.
\end{enumerate}
\end{lemma}
\begin{proof}
As one may use the generalized quasi-Einstein equation and the Bochner formula to establish Assertions~(1)--(3) in
exactly the same fashion that analogous formulas for gradient Ricci almost solitons were established in \cite{LCFLorentzianQE,PRRS};
we omit details in the interests of brevity. One covariantly differentiates Equation~\eqref{eq:general quasi-Einstein} and uses
the definition to establish Assertion~(4). One can express the Cotton tensor in the form: 
\begin{equation}\label{eq:cotton}
\begin{array}{l} 
 C(X,Y,Z)= (\nabla_X \rho)(Y,Z)-(\nabla_Y \rho)(X,Z)\\
\noalign{\medskip}
\phantom{C(X,Y,Z)= -\frac{1}{2}\nabla_X }-\frac{1}{2(n-1)}(X(\tau) g(Y,Z)-Y(\tau) g(X,Z)).
\end{array}
\end{equation} 
We substitute the curvature tensor term into the Weyl tensor and use Assertion~(4). Using Equation~(\ref{eq:cotton}) we may
then make a direct computation to establish Assertion~(5).
\end{proof}

\begin{remark}
\rm
Note that for $\eta=0$ many of the terms in Assertion~(5) of Lemma~\ref{lemma:formulae} vanish. This is precisely the case in which the manifold is conformally Einstein as described previously in Example~\ref{subsubsect:conf-Einstein}.
\end{remark}

\section{The non-isotropic setting: the proof of Theorem~\ref{th:non-isotropic}}\label{S3}
Let $\Lambda^\pm=\{\omega\in\Lambda^2:\star\,\omega=\pm \omega\}$ be the spaces of 
self-dual $(\Lambda^+)$ and anti-self-dual $(\Lambda^-)$ $2$-forms for a 4-dimensional pseudo-Riemannian manifold $\mathcal{M}=(M,g)$.
Let $\{E_1,E_2,E_3,E_4\}$ be an orthonormal local frame for the tangent bundle, let $\varepsilon_i=g(E_i,E_i)=\pm1$ for $i\in\{2,3,4\}$. One has
\begin{eqnarray*}
&&\Lambda^{\pm}=\operatorname{span}\{E^1\wedge E^2\pm \varepsilon_3\varepsilon_4 E^3\wedge E^4,
E^1\wedge E^3\mp \varepsilon_2\varepsilon_4 E^2\wedge E^4,\\
&&\qquad\qquad\qquad E^1\wedge E^4\pm \varepsilon_2\varepsilon_3 E^2\wedge E^3\}\,.
\end{eqnarray*}
We say that $(M,g)$ is self-dual if $W^-=0$. Let $\{i,j,k\}$ be a re-ordering of the indices $\{2,3,4\}$ and let $\sigma_{ijk}$ be the sign of the associated
permutation. Then $\mathcal{M}$ is self-dual if and only if the following identity is satisfied:
\begin{equation}\label{eq:self-dual}
W(E_1,E_i,X,Y)=\sigma_{ijk} \varepsilon_j\varepsilon_k W(E_j,E_k,X,Y) \text{ for any vector fields } X \text{ and } Y.
\end{equation}
We use Assertion~(5) of Lemma~\ref{lemma:formulae}. Let $\eta=2\mu+1$.
We use Equation~\eqref{eq:self-dual} to see that if a quasi-Einstein manifold is self-dual then
\begin{equation}\label{E10}\begin{array}{l}
\phantom{=}\textstyle\frac{\tau {\eta}}{6}\left\{ df(E_i)g(E_1,Z)-df(E_1)g(E_i,Z)\right\}\\[0.05in]
\textstyle\quad+\frac{{\eta}}{6}\{\rho(E_1,\nabla f)g(E_i,Z)-\rho(E_i,\nabla f)g(E_1,Z)\}\\[0.05in]
\textstyle\quad+\frac{{\eta}}{2}\left\{\rho(E_i,Z)df(E_1)-\rho(E_1,Z)df(E_i)\right\}\\[0.05in]
=\textstyle\sigma_{ijk} \varepsilon_j\varepsilon_k \frac{\tau {\eta}}{6}\left\{ df(E_k)g(E_j,Z)-df(E_j)g(E_k,Z)\right\}\\[0.05in]
\textstyle\quad+\sigma_{ijk} \varepsilon_j\varepsilon_k\frac{{\eta}}{6}\{\rho(E_j,\nabla f)g(E_k,Z)-\rho(E_k,\nabla f)g(E_j,Z)\}\\[0.05in]
\textstyle\quad+\sigma_{ijk} \varepsilon_j\varepsilon_k\frac{{\eta}}{2}\left\{\rho(E_k,Z)df(E_j)-\rho(E_j,Z)df(E_k)\right\}\,.
\end{array}\end{equation}
Since $\|\nabla f\|\neq 0$, we may choose the local orthonormal frame so $E_1$ is a non-zero multiple of $\nabla f$.
We may then use Equation~(\ref{E10}) with $Z=E_2$, $Z=E_3$, or $Z=E_4$ to see that $\rho$ is diagonal with respect to this basis.
It now follows that $3\varepsilon_i\rho(E_i,E_i)=\tau-\varepsilon_1 \rho(E_1,E_1)$. The orientation plays no role and the
same conclusion follows if $\mathcal{M}$ is anti-self-dual.
We use Equation~\eqref{eq:general quasi-Einstein} to see that for $2\le i\le 4$ we have:
\[
\Hes_f(E_i,E_i)=\lambda g(E_i,E_i)-\rho(E_i,E_i)=\left(\lambda-\frac{\tau-\varepsilon_1 \rho(E_1,E_1)}3\right) g(E_i,E_i).
\]
Hence, the level hypersurfaces of $f$ are totally umbilical. 
Since $\operatorname{span}\{ \nabla f\}$ is a $1$-dimensional totally geodesic distribution, $\mathcal{M}$ decomposes locally
as a twisted product $I\times_\varphi N$. Since the mixed terms $\rho(E_1,E_i)$ vanish, the twisted product reduces to a warped product.
And, since $I\times_\varphi N$ is self-dual, the warped product is locally conformally flat and the fiber $N$ has constant sectional curvature
(see \cite{Brozos-Garcia-Valle} for a more detailed exposition).
\qed

\section{The isotropic setting I: the proof of Theorem~\ref{th:walker}}\label{S4}

In this section we study isotropic generalized quasi-Einstein manifolds which are half conformally flat and which have neutral
signature $(2,2)$. We fix the orientation so that the manifold is self-dual to simplify the arguments of this section.
We use the fact that $\nabla f$ is a null vector field to choose a local orthonormal frame so that
$$\begin{array}{cc}
-g(E_1,E_1)=g(E_2,E_2)=1,&-g(E_3,E_3)=g(E_4,E_4)=1,\\[0.05in]
g(E_i,E_j)=0\text{ for }i\ne j,&\nabla f=\textstyle\frac1{\sqrt2}(E_1+E_2)\,.
\end{array}$$
We also introduce a corresponding frame of null vector fields
\begin{equation}\label{eq:null-basis}
\mathcal{B}=\left\{\nabla f=\frac{E_1+E_2}{\sqrt2}, U=\frac{E_4-E_3}{\sqrt2},V=\frac{E_2-E_1}{\sqrt2},T=\frac{E_3+E_4}{\sqrt2}\right\}\,.
\end{equation}
This is a hyperbolic frame; the only nonzero components of the metric tensor relative to the local frame $\mathcal{B}$ are
\begin{equation}\label{E13}
g(\nabla f,V)=g(U,T)=1\,.
\end{equation}
We use Equation~\eqref{eq:self-dual} to see that the metric is self-dual if and only if we have the following identities for any $X$ and $Y$:
\begin{eqnarray}
&&W(\nabla f,V,X,Y)=W(U,T,X,Y)\,,\label{eq:self-dual-basisB1}\\
&&W(U,V,X,Y)=0\,,\label{eq:self-dual-basisB2}\\
&&W(\nabla f,T,X,Y)=0\,.\label{eq:self-dual-basisB3}
\end{eqnarray}
\begin{lemma}\label{lemma:isotropic}
Let $(M,g,f,\mu)$ be an isotropic self-dual generalized quasi-Einstein manifold. Then $\lambda=\frac{\tau}{4}$ and the Ricci operator has the form:
\begin{equation}\label{eq:ricci-matrix}
\operatorname{Ric}=\left(\begin{array}{cccc}
\frac{\tau}{4}&0&a&c\\
0&\frac{\tau}{4}&c&b\\
0&0&\frac{\tau}{4}&0\\
0&0&0&\frac{\tau}{4}
\end{array}\right).
\end{equation}
\end{lemma}
\begin{proof}
Since $g(\nabla f,\nabla f)=0$,
$$
0= \nabla_X g(\nabla f,\nabla f)=2g(\nabla_X \nabla f,\nabla f)=2g(\nabla_{\nabla f} \nabla f,X)\text{ for any }X\,.
$$
Consequently, $\hes_f(\nabla f)=0$. Thus by Equation~\eqref{eq:general quasi-Einstein}, $\Ric(\nabla f)=\lambda \nabla f$.
Since $(M,g)$ is self-dual, we take $Y=\nabla f$ in Equation~\eqref{eq:self-dual-basisB1} to see that 
$$
W(\nabla f,V,X,\nabla f)=W(U,T,X,\nabla f)\,.
$$
We then use Assertion~(5) of Lemma~\ref{lemma:formulae} to conclude that $\frac{\tau-4\lambda}6 g(\nabla f, X) =0$ for any $X$.
It now follows that $\lambda=\frac{1}{4}\tau$ as desired. We complete the proof by examining the Ricci tensor.
By Equation~\eqref{eq:self-dual-basisB2}, we have $W(U,V,X,\nabla f)=0$. We use Assertion~(5) of Lemma~\ref{lemma:formulae} to see that
$\rho(U,X)=\textstyle\frac{\tau}{4} g(U,X)$ for all $X$. Consequently, $\Ric(U)=\frac{\tau}{4} U$. We set $X=V$ in Equation~\eqref{eq:self-dual-basisB3} to show that $W(\nabla f,T,V,Y)=0$ for all $Y$. Thus:
$$
\textstyle0=W(Y,V,T,\nabla f)=\frac{1}{2}\left\{\frac{\tau}4 g(Y,T)
-\rho(Y,T) +\rho(V,T)g(Y,\nabla f)\right\}\text{ for all }Y\,.
$$
We set $Y=T$ to see that  $\rho(T,T)=0$ so the Ricci tensor has the form given.
\end{proof}

\medbreak\noindent{\it Proof of Theorem~\ref{th:walker}.}
We have already seen that  $g(\nabla_X \nabla f,\nabla f)=0$. Moreover, a similar argument using the fact that $g(U,U)=0$ shows that
$g(\nabla_X U,U)=0$ for all $X$.  On the other hand, since  $\Ric(U)=\lambda U$, we may use Equation~\eqref{eq:general quasi-Einstein} to see that
$\hes_f(U)=0$. Now, since $g(U,\nabla f)=0$, we have that 
$$
g(\nabla_X U,\nabla f)=-g(U,\nabla_X\nabla f)=-\Hes_f(U,X)=0\text{ for all }X\,.
$$
We have shown that
\begin{equation}\label{E18}\begin{array}{ll}
g(\nabla_X \nabla f,\nabla f)=0,&g(\nabla_X U,U)=0\\[0.05in]
g(\nabla_X U,\nabla f)=0,& g(\nabla_X\nabla f,U)=0\,.
\end{array}\end{equation}
Let $\mathfrak{D}=\operatorname{span}\{\nabla f,U\}$. By Equation~(\ref{E13}), $\mathfrak{D}$ is a null distribution. 
Furthermore, Equation~(\ref{E18}) implies that $\nabla \mathfrak{D}\subset\mathfrak{D}$. Consequently,
 $\mathfrak{D}$ is a $2$-dimensional null parallel distribution and $(M,g)$ is locally a Walker manifold.
\qed

\section{The isotropic setting II: the proof of Theorem~\ref{Thm10} and Theorem~\ref{th:isotropic}}\label{S5}
We continue our study of half conformally flat isotropic generalized quasi-Einstein manifolds by examining the Walker setting. 
The orientation of the manifold, which has not played a role previously, now plays as a role since, as we saw earlier, Walker structures 
have a canonical orientation determined by the null parallel distribution. Adopt the notation of Equation~(\ref{eq:metrictensorofriemannextensions}).

\begin{lemma}\label{lemma:self-dual-Walker}
Let $(M,g)$ be a $4$-dimensional Walker manifold of neutral signature $(2,2)$
which is not Ricci-flat. If $(M,g,f,\mu)$ is a self-dual isotropic generalized quasi-Einstein manifold with $\mu\neq -\frac{1}2$, then 
$(M,g)$ is locally isometric to a modified Riemannian extension $(T^\ast \Sigma,g_{D,\Phi,T,\operatorname{Id}})$ of an affine surface $(\Sigma,D)$  and 
$f=\pi^\ast \hat f$, where $\hat f\in \mathcal{C}^\infty(\Sigma)$.
\end{lemma}
\begin{proof}
We generalize the metric of Equation~(\ref{eq:metrictensorofriemannextensions}) slightly. Results of \cite{CLGRGVL} show that
there exists an affine surface $(\Sigma,D)$, an endomorphism $T$ of the tangent bundle of $\Sigma$, a symmetric bilinear form $\Phi$
on the tangent bundle of $\Sigma$, and a vector field $X$ on $\Sigma$ so that $\mathcal{M}$ is locally isometric to $(T^*\Sigma,g)$ where
\begin{equation}\label{eq:modified-Riemannian-extension}
g=\iota X(\iota \Id\circ\iota \Id)+ \iota T\circ\iota\Id +g_{D,\Phi}\,.
\end{equation}
As case $\mu=0$ was considered previously in \cite{Brozos-Garcia-Valle}, we shall assume that $\mu\neq 0$. 
We first show that $f=\pi^\ast \hat f$ for some $\hat f\in \mathcal{C}^\infty(\Sigma)$.  We set $h=e^{-\mu f}$ in Equation~\eqref{eq:general quasi-Einstein} to obtain 
the equivalent equation
$$
-\operatorname{Hes}_{h}+\mu {h}\rho=\mu {h} \lambda g\,.
$$
A similar change of variable will play a central role in the discussion of Section~\ref{S6} as well.
Lemma~\ref{lemma:isotropic} shows that $\lambda=\frac{\tau}{4}$. Consider the symmetric $(0,2)$-tensor
\begin{equation}\label{eq:def-qee-app}
\textstyle\mathfrak{Q}(h):=-\operatorname{Hes}_{h}+\mu {h}\left(\rho- \frac{\tau}{4} g\right)\,.
\end{equation}
For $i=1,2$, we compute  that
\begin{equation}\label{eq:pullback-f}
\mathfrak{Q}(h)(\partial_{x_{i'}},\partial_{x_{j'}}) =-\partial^2_{x_{i'}x_{j'}} h =0\,.
\end{equation}
Therefore, there exists a vector field $\xi$ on $\Sigma$ and a smooth function $\hat h$ on $\Sigma$ so that
\begin{equation}\label{gr3}
h=\iota\xi+\pi^*\hat{h}\,.
\end{equation}
If we can show that $\xi=0$, then it will follow that $h=\pi^*\hat h$ and correspondingly that $f=\pi^*\hat f$ as desired. Suppose to the contrary that
$\xi\ne0$. We argue for a contradiction. Choose local coordinates on $\Sigma$ so $\xi=\partial_{x^1}$. By Equation~(\ref{gr3}), $h=x_{1'}+\pi^*\hat{h}$.  
If $P$ is a polynomial in certain variables, let $\coef(P;\cdot)$ denote the coefficient of a given term in $P$. Adopt the
notation of Equation~(\ref{eq:modified-Riemannian-extension}). Expand $X=X_1\partial_{x^1}+X_2\partial_{x^2}$. 
We have $\mathfrak{Q}(h)=0$. Since $\mu\ne-\frac12$, we may show that $X=0$ by computing:
\begin{eqnarray*}
&&\coef(\mathfrak{Q}(h)(\partial_{x^{2}},\partial_{x_{1'}}));x_{1'} x_{2'})=(1+2\mu)X_1,\\
&&\textstyle\coef(\mathfrak{Q}(h)(\partial_{x^{2}},\partial_{x_{1'}});x_{2'}^2)=\frac{1}{2}X_2\,.
\end{eqnarray*}
Let $T=(T_i^j)$. Since $\mu\ne0$ and $\mu\ne-\frac12$, we show similarly that $T=0$ by computing:
\begin{eqnarray*}
&&\textstyle\coef(\mathfrak{Q}(h)(\partial_{x^{1}},\partial_{x_{1'}});x_{2'})=\frac12 T_1^2,\\
&&\textstyle\coef(\mathfrak{Q}(h)(\partial_{x^{2}},\partial_{x_{1'}});x_{1'})=\frac12 (1+2\mu) T_2^1,\\
&&\textstyle\coef(\mathfrak{Q}(h)(\partial_{x^{2}},\partial_{x_{1'}});x_{2'})=\frac14 (T_1^1+T_2^2),\\
&&\textstyle\coef(\mathfrak{Q}(h)(\partial_{x^{2}},\partial_{x_{2'}});x_{1'})=\frac14 \{(1-2\mu) T_1^1+(1+2\mu)T_2^2\}\,.
\end{eqnarray*}
Because $\lambda=\frac{\tau}4=\frac34 (T_1^1 + T_2^2 + 4 X_1 x_{1'}+ 4 X_2 x_{2'})$, we have $\lambda=0$. 
We compute:
\begin{eqnarray*}
&&\mathfrak{Q}(h)(\partial_{x^{1}},\partial_{x_{1'}})=- \Gamma_{11}{}^1,\quad
\mathfrak{Q}(h)(\partial_{x^{1}},\partial_{x_{2'}})=- \Gamma_{11}{}^2,\\
&&\mathfrak{Q}(h)(\partial_{x^{2}},\partial_{x_{1'}})=- \Gamma_{12}{}^1,\quad
\mathfrak{Q}(h)(\partial_{x^{2}},\partial_{x_{2'}})=- \Gamma_{12}{}^2,\\
&&\coef(\mathfrak{Q}(h)(\partial_{x^{2}},\partial_{x^{2}});x_{1'})=(1+2\mu)\partial_{x^1} \Gamma_{22}{}^1,\\
&&\coef(\mathfrak{Q}(h)(\partial_{x^{2}},\partial_{x^{2}});x_{2'})=\partial_{x^1} \Gamma_{22}{}^2.
\end{eqnarray*}
Hence setting $\mathfrak{Q}(h)=0$,
$$
\partial_{x^1}\Gamma_{22}{}^1=\partial_{x^1}\Gamma_{22}{}^2=0\text{ and }
\Gamma_{11}{}^1=\Gamma_{11}{}^2= \Gamma_{12}{}^1=\Gamma_{12}{}^2=0\,.
$$
This implies that $\rho=0$ which is contrary to our assumption.
Consequently, as desired
$$
\xi=0\text{ so }h=\pi^*\hat{h}\,.
$$
We expand $X=X^1 \partial_{x^{1}}+X^2 \partial_{x^{2}}$ and compute: 
\begin{eqnarray*}
&&\coef(\mathfrak{Q}(h)(\partial_{x^{1}},\partial_{x_{1'}});x_{1'})=\mu h X^1 h,\\
&&\coef(\mathfrak{Q}(h)(\partial_{x^{2}},\partial_{x_{2'}});x_{1'})=\mu h X^2 h\,.
\end{eqnarray*}
This shows that $X=0$ and, as desired, $g=\iota T\circ\iota\Id +g_{D,\Phi}$.
\end{proof}

The following is an example where $(T^\ast\Sigma, g_{D,\Phi,T,\Id})$ is a conformally Einstein modified Riemannian extension (i.e. a generalized quasi-Einstein manifold with $\mu=-\frac12$) where
the conformal function is not the pull-back of a function on the surface $(\Sigma,D)$. Consequently the assumption that 
$\mu\ne-\frac12$ in Lemma~\ref{lemma:self-dual-Walker} is essential.
\begin{example}\rm\label{remark:conf-Einstein}
Let $(x^1,x^2)$ be the usual coordinates on $\Sigma=\mathbb{R}^2$. Suppose given smooth functions
$\alpha(x^1,x^2)$, $\beta(x^1,x^2)$, and $\gamma(x^1,x^2)$ where $\gamma\ne0$. Also suppose given smooth functions $\psi_1(x^2)$
and $\psi_2(x^2)$. We consider the following structures defining $D$, $T$, $f$, and $\Phi$:
$$\begin{array}{ll}
\Gamma_{12}{}^1=\alpha,&\Gamma_{22}{}^1=\beta\\[0.05in]
\Gamma_{22}{}^2=\alpha+\psi_1(x^2),&T=\gamma dx^2\otimes \partial_{x^1},\\
f(x^1,x^2,x_{1'},x_{2'})=x_{1'}-2 \frac\alpha \gamma,&\Phi_{11}=-4\, \partial_{x^1}\!\!\left(\frac{\alpha}{\gamma}\right),\\[0.05in]
\Phi_{12}=\Phi_{21}=\frac{2\alpha^2}{\gamma}-\partial_{x^2}\!\!\left(\frac{2\alpha}{\gamma}\right),&\Phi_{22}=\frac{4 \alpha\beta}{\gamma}+\psi_2(x^2)\,.
\end{array}$$
One then has that $(T^\ast\Sigma, g_{D,\Phi,T,\Id},f,-\frac{1}{2})$ is generalized quasi-Einstein. The Ricci tensor is always nonzero and two-step nilpotent
and the conformal function has a null gradient.
\end{example}

\subsection*{The proof of Assertion~(1) of Theorem~\ref{Thm10} and Theorem~\ref{th:isotropic}}
Let $(M,g,f,\mu)$ be a self-dual generalized quasi-Einstein  manifold of signature $(2,2)$ with $\mu\neq -\frac{1}{2}$ and $\|\nabla f\|=0$
which is not Ricci flat.  Suppose that $\lambda$ is constant. 
By Lemma~\ref{lemma:isotropic}, $\tau=4\lambda$. We may now use Assertion~(3) of Lemma~\ref{lemma:formulae} to see that $\lambda=0$. 

Lemma~\ref{lemma:self-dual-Walker} shows that $g=g_{D,\Phi,T,\operatorname{Id}}$ is locally isometric to a modified Riemannian extension.
Adopt the notation of the proof of Lemma~\ref{lemma:self-dual-Walker}. We compute that
$$\begin{array}{cc}
\mathfrak{Q}(h)(\partial_{x^{1}},\partial_{x_{1'}})=\mu h T_1^1,&
\mathfrak{Q}(h)(\partial_{x^{1}},\partial_{x_{2'}})=\mu h T_1^2,\\[0.05in]
\mathfrak{Q}(h)(\partial_{x^{2}},\partial_{x_{1'}})=\mu h T_2^1,&
\mathfrak{Q}(h)(\partial_{x^{2}},\partial_{x_{2'}})=\mu h T_2^2.
\end{array}
$$
Since $\mu\ne0$, since $h\ne0$, and since $\mathfrak{Q}(h)=0$, we have $T=0$. This shows that $g$ is indeed locally isometric to a 
deformed Riemannian extension $g_{D,\Phi}$ as introduced in Subsection~\ref{sect:Walker}. 

We know from Lemma~\ref{lemma:self-dual-Walker} that $f=\pi^\ast \hat f$. Since $\lambda=0$, the generalized quasi-Einstein equation reduces to the quasi-Einstein equation $\operatorname{Hes}_{f}+\rho-\mu df\otimes df=0$. A direct computation shows that the only non-vanishing terms of this equation are 
\[
(\operatorname{Hes}_{f}+\rho-\mu df\otimes df)(\partial_{x^{i}},\partial_{x^{j}})=(\operatorname{Hes}^D_{\hat f}+2\rho^D_s-\mu d\hat f\otimes d\hat f)(\partial_{x^{i}},\partial_{x^{j}}),
\]
for $i,j=1,2$.
Thus the manifolds of Assertion~(1) of Theorem~\ref{Thm10} are indeed quasi-Einstein. Furthermore, as
 stated in Assertion (1) of Theorem~\ref{th:isotropic}, these are the only examples.
\qed

\subsection*{The proof of Assertion~(2) of Theorem~\ref{Thm10} and Theorem~\ref{th:isotropic}}
Let $(M,g,f,\mu)$ be a self-dual generalized quasi-Einstein  manifold of signature $(2,2)$ with $\mu\neq -\frac{1}{2}$ and $\|\nabla f\|=0$
which is not Ricci flat.  Suppose that $\lambda$ is non-constant.
By Lemma~\ref{lemma:self-dual-Walker}, $\mathcal{M}$ is locally isometric to a modified Riemannian extension of the form 
$g_{D,\Phi,T,\operatorname{Id}}$ with $f=\pi^\ast \hat f$.
 Since the case $\mu=0$ was already studied in \cite{Brozos-Garcia-Valle}, we suppose $\mu\ne0$ as well. Once again, we make
 the change of variable $h=e^{-\mu f}$ and work with the symmetric bilinear form $\mathfrak{Q}(h)$ of \eqref{eq:def-qee-app}.
We compute:
\begin{eqnarray*}
&&\textstyle\mathfrak{Q}(h)(\partial_{x^{1}},\partial_{x_{1'}})=\frac12 \,\mu\, h\, (T_1^1-T_2^2),\\
&&\textstyle\mathfrak{Q}(h)(\partial_{x^{1}},\partial_{x_{2'}})=\mu\, h\, T_1^2,\quad\text{and}\quad
\mathfrak{Q}(h)(\partial_{x^{2}},\partial_{x_{1'}})=\mu \,h\, T_2^1\,.
\end{eqnarray*}
Since $\mathfrak{Q}(h)=0$,  $T=\mathfrak{t}\, \Id$ is a multiple of the identity and consequently
$$
g=\pi^*(\mathfrak{t})\,\iota \Id\circ\iota\Id +g_{D,\Phi} \text{ for some }\mathfrak{t}\in \mathcal{C}^\infty(\Sigma)\,.
$$
We once again work with the generalized quasi-Einstein Equation~\eqref{eq:general quasi-Einstein} to compute
\begin{eqnarray*}
&&\coef((\Hes_f+\rho-\mu df\otimes df-\lambda g)(\partial_{x^{1}},\partial_{x^{1}});x_{1'})=\mathfrak{t} \partial_{x^1} f+\partial_{x^1}\mathfrak{t},\\
&&\coef((\Hes_f+\rho-\mu df\otimes df-\lambda g)(\partial_{x^{2}},\partial_{x^{2}});x_{2'})=\mathfrak{t} \partial_{x^2} f+\partial_{x^2}\mathfrak{t}\,.
\end{eqnarray*}
Setting these terms to zero yields $\mathfrak{t}=C e^{-\hat f}$. One verifies that
\begin{eqnarray*}
&&(\Hes_f+\rho-\mu df\otimes df-\lambda g)(\partial_{x^{i}},\partial_{x^{j}})\\
&=&\textstyle\frac{C}2 e^{-\hat f} \Phi_{ij}-(\operatorname{Hes}^D_{\hat f}+2\rho^D_s-\mu d\hat f\otimes d\hat f)(\partial_{x^{i}},\partial_{x^{j}})
\text{ for }i=1,2\,.
\end{eqnarray*}
This shows that the construction of Assertion~(2) of Theorem~\ref{Thm10}
 provides examples of generalize quasi-Einstein manifolds; Assertion~(2) of Theorem~\ref{th:isotropic} follows.
\qed 

 
\begin{remark}\rm\label{remark:anti-self-dual-Walker}
Let $(M,g)$ be a Walker manifold of signature $(2,2)$. If $(M,g,f,\mu)$ is an anti-self-dual isotropic generalized quasi-Einstein manifold then one still has that $\lambda =\frac{\tau}{4}$. To see this, one can proceed as in the proof of Lemma~\ref{lemma:isotropic} and use the anti-self-dual relation $W(\nabla f,V,X,Y)=-W(U,T,X,Y)$ instead of \eqref{eq:self-dual-basisB1}.
It was shown in \cite{DR-GR-VL} that if the self-dual Weyl curvature $W^+$ of a Walker manifold vanishes, then $\tau=0$, so $\lambda=0$ and there are not generalized quasi-Eisntein examples with non-constant $\lambda$.
We refer to Section \ref{S9} for some explicit examples of anti-self-dual quasi-Einstein manifolds which are not realized as a deformed Riemannian extension as in Theorem \ref{th:isotropic} (1).
These differences between the self-dual and the anti-self-dual conditions illustrate the fact that the Walker structure determines the orientation and, hence, self-duality and anti-self-duality are not interchangeable conditions in Walker geometry.
\end{remark}

\section{The affine quasi-Einstein equation}\label{S6}

Let $(M,g,f,\mu)$ be a self-dual generalized quasi-Einstein manifold of signature $(2,2)$ with $\lambda$ constant which is
not Ricci flat. Assume $\mu\ne-\frac12$ and $\|\nabla f\|=0$. By Theorem~\ref{th:isotropic}, $\lambda=0$ and $(M,g,f,\mu)$ is locally isomorphic to
$(T^*\Sigma,g_{D,\phi},f,\mu)$ where $f=\pi^*\hat f$ and $\hat f$ satisfies
$\operatorname{Hes}^D_{\hat f}+2\rho^D_s-\mu d\hat f\otimes d\hat f=0$.
Conversely, by Theorem~\ref{Thm10}, every such example is a self-dual isotropic quasi-Einstein Walker manifold with $\lambda=0$. Thus it
is natural to consider this equation in its own right on affine surfaces. If $\mu=0$, then $(\Sigma,D,\hat f)$ is an affine gradient Ricci soliton as 
discussed in \cite{MeE}. Thus we shall suppose that $\mu\ne0$ and consider the change of variables
$\hat h=e^{-\frac12\mu\hat f}$; a similar change of variables played a crucial role in the analysis of
Section~\ref{S5}. Equation~(\ref{eq:general quasi-Einstein}) then becomes $-\frac{2}{\mu\hat h}\operatorname{Hes}^D_{\hat h}+2\rho^D_s=0$.
This leads to the {\it affine quasi-Einstein equation}
\begin{equation}\label{QEE}
\operatorname{Hes}^D_{\hat h}=\mu\, \hat h\,\rho^D_s\text{ for }\mu\in\mathbb{R}\,.
\end{equation}
Let $\mathcal{M}:=(M,D)$ be an affine manifold. In Equation~(\ref{QEE}), the {\it eigenvalue} $\mu$ is a parameter that needs to be determined. Let
$E_{\mathcal{S}}(\mu)=E(\mu)$ be the vector space of smooth solutions to the linear partial differential Equation~\eqref{QEE}: 
$$
E(\mu):=\{\hat h\in \mathcal{C}^\infty( M):\operatorname{Hes}^D_{\hat h}=\mu\, \hat h\,\rho^D_s\}\,.
$$
Similarly, if $p$ is a point of $M$, let $E(p,\mu)$ be the linear space of all germs of solutions to Equation~\eqref{QEE} based at $p$. 
Let $\mathfrak{A}(p)$ be the Lie algebra of germs of affine Killing vector fields based at $p$.
We summarize as follows some results concerning this equation and refer to a subsequent paper \cite{BGGV16} for the proof. We
work in complete generality and do not restrict to the case of surfaces for the moment.

\begin{theorem}\label{th:affine}
Let $\mathcal{M}$ be an affine manifold of dimension $n$. Let $p\in M$.
\begin{enumerate}
\item If $\hat h$ is a $\mathcal{C}^2$ solution to Equation~(\ref{QEE}), then $\hat h$ is in fact smooth.
If $(M,D)$ is real analytic, then $\hat h$ is real analytic.
\item Let $\hat h\in E(p,\mu)$. If $\hat h(p)=0$ and if $d\hat h(p)=0$, then $\hat h\equiv0$.
\item $\dim\{E(p,\mu)\}\le n+1$.
\item If $X\in\mathfrak{A}(p)$ and if $\hat h\in E(p,\mu)$, then $X\hat h\in E(p,\mu)$.
\item If $\Sigma$ is simply connected and if $\dim\{E(p,\mu)\}$ is constant on $\Sigma$, then any element $\hat h\in E(p,\mu)$ extends
uniquely to an element of $E(\mu)$.
\end{enumerate}\end{theorem}

The extremal case where $\dim\{E(p,\mu)\}=n+1$ merits additional attention.
\begin{definition}\label{D21}\rm We say that $D$ is {\it projectively flat} if there exists a $1$-form $\omega$ and a flat connection $\tilde D$ so that
$D_XY=\tilde D_XY+\omega(X)Y+\omega(Y)X$ for all $X$ and $Y$; $D$ is projectively
flat if and only if it is possible to choose a coordinate system so that the unparametrized geodesics of $D$ are straight lines.
We say that $D$ is {\it strongly projectively flat} if in addition $\omega$ can be chosen to be closed. 
\end{definition}
The eigenvalue $\mu_n=-\frac1{n-1}$ is distinguished in this subject. It appears, for example, in Example \ref{subsubsect:conf-Einstein}.
The following result relates the analytic properties of the affine quasi-Einstein equation to the underlying affine geometry.

\begin{theorem}\label{T22}
Let $\mathcal{M}$ be an affine manifold of dimension $n$. Let $\mu_n:=-\frac1{n-1}$.
\begin{enumerate}
\item $\mathcal{M}$ is  strongly projectively flat if and only if $\dim\{E(\mu_n)\}=n+1$.
\item If $\dim\{E(\mu)\}=n+1$ for any $\mu$, then $\mathcal{M}$ is strongly projectively flat.
\item If $\dim\{E(\mu)\}=n+1$ for $\mu\ne\mu_n$, then $\mathcal{M}$ is Ricci flat.
\item Suppose $\dim\{E(p,\mu_n)\}=n+1$. One may choose a basis $\{\phi_0,\dots,\phi_n\}$ for $E(p,\mu_n)$ so that
$\phi_0(p)\ne0$ and $\phi_i(p)=0$ for $i>0$. Set $z^i:=\phi_i/\phi_0$. Then $\vec z=(z^1,\dots,z^n)$
is a system of coordinates defined near $p$ such that the unparametrized geodesics of $\mathcal{M}$ are  straight lines.
\end{enumerate}\end{theorem}

We remark that if $\rho_s^D=0$, then $E(\mu)=E(0)$ for any $\mu$. The space $E(0)$ is the space of Yamabe solitons.

\section{Homogeneous surface geometries}\label{S7}
In this section, we examine solutions to the affine quasi-Einstein equation in the context of homogeneous affine surfaces.
Since an affine surface is flat if and only if the Ricci tensor vanishes, all the geometry is encoded in the Ricci tensor. In particular,
a geometry is symmetric if and only if $D\rho^D=0$.

We say that $\mathcal{S}=(\mathbb{R}^2,D)$ is a {\it Type~$\mathcal{A}$ surface model} if the Christoffel symbols $\Gamma_{ij}{}^k$ of
the connection $D$ are constant. Similarly, we say that $\mathcal{S}=(\mathbb{R}^+\times\mathbb{R},D)$ is a {\it Type~$\mathcal{B}$ surface model} if the
Christoffel symbols of the connection $D$ have the form $\Gamma_{ij}{}^k=(x^1)^{-1}C_{ij}{}^k$ for $C_{ij}{}^k$ constant. 
The Ricci tensor of any Type~$\mathcal{A}$ surface model is symmetric; this can fail for Type~$\mathcal{B}$ surface models. Any Type~$\mathcal{A}$
surface model is projectively flat; this can fail for Type~$\mathcal{B}$ surface models. These geometries are
homogeneous; the translations $(x^1,x^2)\rightarrow(x^1+a^1,x^2+a^2)$ act transitively on $\mathbb{R}^2$ and preserve any Type~$\mathcal{A}$
connection. Similarly, the coordinate transformations $(x^1,x^2)\rightarrow(ax^1,ax^2+b)$ for $a>0$ act transitively on $\mathbb{R}^+\times\mathbb{R}$
and preserve any Type~$\mathcal{B}$ connection. Since the geometries are homogeneous, $\dim\{E(p,\mu)\}$ is constant. Since the underlying
topological space is simply connected, we may use Theorem~\ref{th:affine}~(5) to identify the global solutions $E(\mu)$ with the germs of local
solutions $E(p,\mu)$ for any point $p$. Thus for these geometries, there is no difference between the global and the local theory.

The importance of these two geometries lies in the fact that Opozda~\cite{Op04} showed that any locally homogeneous affine surface which is not flat
is modeled on a Type~$\mathcal{A}$ geometry, on a Type~$\mathcal{B}$ geometry, or on the Levi-Civita connection of the sphere.
These categories are not disjoint. A Type~$\mathcal{B}$ surface model $\mathcal{S}$ is locally isomorphic to a
Type~$\mathcal{A}$ surface model if and only if $C_{12}{}^1=C_{22}{}^1=C_{22}{}^2=0$; we refer to \cite{BGG16} for details. For surfaces, the critical
eigenvalue is $\mu_2=-\frac1{n-1}|_{n=2}=-1$. Since any Type~$\mathcal{A}$ surface model is strongly projectively flat, $\dim\{E(-1)\}=3$ 
by Theorem~\ref{T22}. As we assumed that $\mathcal{S}$ is not flat, $\mathcal{S}$ is not Ricci flat and consequently $\dim\{E(\mu)\}\le 2$ for $\mu\ne-1$.

If $\mathcal{S}$ is a Type~$\mathcal{A}$ surface model, then $\partial_{x^1}$ and $\partial_{x^2}$ are affine Killing vector fields and thus $E(\mu)$ is a 
$\operatorname{Span}\{\partial_{x^1},\partial_{x^2}\}$ module by Theorem~\ref{th:affine}~(4). In the Type~$\mathcal{B}$ setting,
$x^1\partial_{x^1}+x^2\partial_{x^2}$ and $\partial_{x^2}$ are affine Killing vector fields and thus
$E(\mu)$ is a $\operatorname{Span}\{x^1\partial_{x^1}+x^2\partial_{x^2},\partial_{x^2}\}$ module. These module structures play a crucial role
in the proof given in \cite{BGGV16} of the following results:

\begin{theorem}
Let $\mathcal{S}$ be a Type~$\mathcal{A}$ surface model. Let $\mu\ne0$.
$$\dim\{E(\mu)\}=\left\{\begin{array}{ll}
3\text{ if }\mu=-1\\
2\text{ if }\mu\ne-1\text{ and }\operatorname{Rank}\{\rho^D\}=1\\
0\text{ if }\mu\ne-1\text{ and }\operatorname{Rank}\{\rho^D\}=2
\end{array}\right\}\,.$$
\end{theorem}

\begin{theorem}\label{T24}
Let $\mathcal{S}$ be a Type $\mathcal{B}$ surface model which is not also Type~$\mathcal{A}$. 
If $E(\mu)\ne\{0\}$, then up to linear equivalence, one of the following holds:
\begin{enumerate}
\item $C_{22}{}^1=\pm1$, $C_{12}{}^1=0$, $C_{22}{}^2=\pm2C_{11}{}^2$, $\Delta:=-C_{11}{}^1+C_{12}{}^2+1\ne0$, 
$\mu=\Delta^{-2}\{-(C_{11}{}^1)^2+2 C_{11}{}^1 C_{12}{}^2+2 (C_{11}{}^2)^2-(C_{12}{}^2)^2+2 C_{12}{}^2+1\}$.
\item $C_{22}{}^1=0$, $C_{12}{}^1=C_{22}{}^2\neq 0$, $\mu=-1$.
\end{enumerate}
\end{theorem}

Undoing the transformation from Equation~\eqref{eq:QEE-affine} to Equation~\eqref{QEE}, we see that $\mu=-1$ corresponds to the conformally Einstein case
of Example~\ref{subsubsect:conf-Einstein}. We give the following example to illustrate Theorem~\ref{T24} and we refer to \cite{BGGV16} for an explicit
description of the relevant eigenspaces; we will generalize this example subsequently in Section~\ref{S8}.

We consider the following family of Type~$\mathcal{B}$ surface models. It is convenient to change notation slightly to simplify the computations:

\begin{definition}\label{E25}\rm
Let $S:=\{(x^1,x^2)\in\mathbb{R}^2:x^1+x^2>0\}$. For $0\ne\kappa\in\mathbb{R}$, let $\mathcal{S}_\kappa=(S,D_\kappa)$ be a Type~$\mathcal{B}$ surface
model where the nonzero Christoffel symbols of $D_\kappa$ are given by taking $\Gamma_{11}{}^1=\Gamma_{22}{}^2=\kappa\,(x^1+x^2)^{-1}$. Let
$E_\kappa(\mu)$ be the associated eigenspace of the affine quasi-Einstein equation on $\mathcal{S}_\kappa$.
\end{definition}

\begin{theorem}\label{T26}  The Ricci tensor of $\mathcal{S}_k$ is recurrent. The affine surface $\mathcal{S}_k$
is a symmetric space if and only if $\kappa=-2$. One has:
\begin{enumerate}
\item $E_\kappa(0)=\mathbb{R}$.
\item If $\kappa=-2$ and $\mu=-1$, then $E_\kappa(\mu)=(x^1+x^2)^{-1}\operatorname{Span}\{1,x^1-x^2\}$.
\item If $\kappa\ne-2$ and $\mu=\kappa+1$, then $E_\kappa(\mu)=(x^1+x^2)^{\kappa+1}\cdot\mathbb{R}$.
\item If $\mu\ne0$ and $\mu\ne\kappa+1$, then $E_\kappa(\mu)=\{0\}$.
\end{enumerate}\end{theorem}

\begin{proof}  We note that $\mathcal{S}_{-2}$ is isomorphic to the Lorentzian hyperbolic plane. We show that the Ricci tensor is recurrent and that
$\mathcal{S}_\kappa$ is symmetric if and only if $\kappa=-2$ by computing:
\[
\rho^D=\frac{\kappa}{(x^1+x^2)^{2}}\left(\begin{array}{ll}0&1\\1&0\end{array}\right),\quad
D\rho^D=-\frac{2+\kappa}{x^1+x^2}(dx^1+dx^2)\otimes\rho^D\,.
\]
One can use the structure of $E_\kappa(\mu)$ as a
module over the Lie algebra of affine Killing vector fields to show that if  $E_\kappa(\mu)$ is non-trivial, then there exists $\alpha\in\mathbb{C}$ so that
$\hat h:=(x^1+x^2)^\alpha\in E_\kappa(\mu)$. We compute
$$
\Hes^D_{\hat h}-\mu \hat h \rho^D= (x^1+x^2)^{\alpha-2}\left(
\begin{array}{cc}
\alpha  (\alpha -\kappa -1) & \alpha ^2-\alpha -\mu  \kappa  \\
\alpha ^2-\alpha -\mu  \kappa  & \alpha  (\alpha -\kappa -1) \\
\end{array}
\right)\,.
$$
Since $\kappa\ne0$, if $\mu\ne0$ then $\alpha\ne0$ and thus $\alpha=\mu=\kappa+1$. So $E_\kappa(0)$ and $E_\kappa(\kappa+1)$ are the only
non-trivial eigenspaces; this is in agreement with Theorem~\ref{T24} where there is at most 1 non-trivial eigenvalue $\mu\ne0$.
The module structure is used in \cite{BGGV16} to find the general form of an element of $E(\mu)$ for Type~$\mathcal{A}$ and Type~$\mathcal{B}$ surface
models.  When those results are applied to the setting at hand, we may conclude that $E_\kappa(\mu)$ is spanned by elements of the form:
$$
\hat h(x^1,x^2)=(x^1+x^2)^\beta\{c_1(x^1+x^2)+c_2(x^1-x^2)+c_3(x^1+x^2)\log(x^1+x^2)\}\,.
$$
A computer aided calculation yields that $c_3=0$. Furthermore, $\beta=\kappa+1$,  if $c_2\ne0$, and $\beta=\kappa$, if $c_2=0$. A careful analysis of the different cases then yields the remainder of Theorem~\ref{T26}.\end{proof}

\section{Inhomogeneous examples}\label{S8}
In Section~\ref{S7}, we exhibited a number of homogeneous examples illustrating different phenomena related to the
affine quasi-Einstein equation (Equation~\eqref{QEE}). In this section, we exhibit some inhomogeneous examples.
We begin with a useful ansatz; we shall suppose $n=2$ but it is valid for general $n$. Let $\phi_i:=\partial_{x^i}\phi$, $\phi_{ij}:=\partial_{x^j}\phi_i$, etc.
We omit the details of the following computation as it is entirely elementary:
\begin{lemma}\label{L27}
Let $\mathcal{O}$ be a simply connected open subset of $\mathbb{R}^2$. Let $\phi\in \mathcal{C}^\infty(\mathcal{O})$ and let $\mathcal{M}_\phi=(\mathcal{O},D^\phi)$
where $D^\phi$ is the affine connection on $\mathcal{O}$ with $\Gamma_{ii}{}^i=\phi_{ii}/\phi_i$ and $\Gamma_{ij}{}^k=0$.
Then $\phi\in E(\mu)$ if and only if $\phi$ satisfies the non-linear partial differential equation:
\begin{equation}\label{E24}
0=\frac{1}{2} \mu  \phi  \left(\frac{\phi_{122}}{\phi_{2}}+\frac{\phi_{112}}{\phi_{1}}\right)
+\phi_{12} \left(1-\frac{1}{2} \mu  \phi  \left(\frac{\phi_{22}}{\phi_{2}^2}+\frac{\phi_{11}}{\phi_{1}^2}\right)\right)\,.
\end{equation}
\end{lemma}

We will use the Cauchy-Kovalevskaya Theorem (see, for example, Evans \cite{Evans}) to construct solutions to Equation~(\ref{E24}) and thereby show
Lemma~\ref{L27} is non-trivial. Fix a point $p\in\mathbb{R}^2$. Let $c=\phi(p)$, $c_i=\phi_i(p)$, and so forth.
Let $\vec{c}:=\mathcal{J}_k(\phi)(p)$ be the $k$-jet of $\phi$ at $p$. For example
$$
\mathcal{J}_3=\{\vec c=(c,c_1,c_2,c_{11},c_{12},c_{22},c_{111},c_{112},c_{122},c_{222})\}\subset\mathbb{R}^{10}\,.
$$
Let $\tilde{\mathcal{J}}_k$ be the subset of $\mathcal{J}_k$ where $c\ne0$, $c_1\ne0$, $c_2\ne0$, and where the relations imposed
by differentiating Equation~(\ref{E24}) are satisfied. For example, $\tilde{\mathcal{J}}_2$ is the open dense subset of $\mathbb{R}^6$
with 8 path components obtained by deleting 3 hyperplanes. If we fix an element $\xi_2\in\tilde{\mathcal{J}}_2$, then the relation of Equation~(\ref{E24}) is
linear in the third derivatives of $\phi$ and thus the natural projection $\tilde{\mathcal{J}}_3\rightarrow\tilde{\mathcal{J}}_2$ is a real analytic
3-dimensional vector bundle.  If we fix an element $\xi_3\in\tilde{\mathcal{J}}_3$, there are two linearly independent relations in the 4-jets of $\phi$ which arise
from differentiating Equation~(\ref{E24}) and the natural projection $\tilde{\mathcal{J}}_4\rightarrow\tilde{\mathcal{J}}_3$ is again a 3-dimensional real analytic
vector bundle. Arguing similarly, we see that $\tilde{\mathcal{J}}_k$ is a real analytic manifold of dimension $3k$. Assume that the symmetric Ricci tensor 
is non-degenerate. Let $\rho_{ij}$ be the components of $\rho^D_s$, let $\rho^{ij}$ be the components of the inverse matrix, and let $\rho_{ij;k}$ be the
components of the covariant derivatives of the symmetric Ricci tensor.
We define a scalar invariant $\mathcal{E}$ of such a geometry by setting
$$
\mathcal{E}:=\rho^{ia}\rho^{jb}\rho^{kc}\rho_{ij;k}\rho_{ab;c}\,.
$$
\begin{theorem}
\ \begin{enumerate}
\item Given $\xi\in\tilde{\mathcal{J}}_k$, there exists the germ of a function $\phi$ with $ \mathcal{J}_k(\phi)=\xi$ solving Equation~(\ref{E24}).
\item If $k\ge3$, there is an open dense subset $\mathcal{O}_k$ of $\tilde{\mathcal{J}}_k$ so that if $ \mathcal{J}_k(\phi)\in\mathcal{O}_k$,
then the Ricci tensor is not symmetric.
\item Let $\mu\ne-1$. If $k\ge5$, there is an open dense subset $\mathcal{O}_k$ of $\tilde{\mathcal{J}}_k$ so that if $J_k(\phi)\in\mathcal{O}_k$,
then $\rho_s^D$ is non degenerate, $d\mathcal{E}\ne0$, and the geometry is not homogeneous.
\end{enumerate}\end{theorem}

\begin{proof} We recall the classical Cauchy-Kovalevskaya Theorem. Suppose we are given a relation
in the 3-jets of $\phi$ which is linear once the 2-jets have been fixed and has the form:
\begin{equation}\label{E25a}
\begin{array}{l}
0=\alpha_{111}( \mathcal{J}_2(\phi))\phi_{111}+\alpha_{112}(\mathcal{J}_2(\phi))\phi_{112}+\alpha_{122}(\mathcal{J}_2(\phi))\phi_{122}\\[0.05in]
\phantom{0=}+\alpha_{222}(\mathcal{J}_2(\phi))\phi_{222}+\alpha(\mathcal{J}_2(\phi)).
\end{array}\end{equation}
Assume the coefficient functions are real analytic on some suitable open set of $\mathcal{J}_2$. Suppose that $\alpha_{111}\ne0$. Then given
Cauchy data $f_0(x^2)$, $f_1(x^2)$, and $f_2(x^2)$, there is a unique real analytic solution to Equation~\eqref{E25a} with $\partial_{x^1}^i\phi(0,x^2)=f_i(x^2)$
for $i=0,1,2$. If one expands $\phi$ in a Taylor series, the derivatives $\partial_{x^1}^i\partial_{x^2}^j\phi(0,0)$ can be specified arbitrarily for $i\le2$ and the
remaining Taylor series coefficients are then determined by Equation~(\ref{E25a}). Reinterpreting this in the language we have introduced, this
means that if $\xi\in\tilde{\mathcal{J}}_k$ is given, there exists a solution $\phi$ with $J_k(\phi)=\xi$. This observation is not directly applicable to
the setting at hand since $\alpha_{111}=0$ for our example. But since we have assumed $\phi(0)\ne0$, $\phi_1(0)\ne0$, and $\phi_2(0)\ne0$,
Equation~(\ref{E24}) does provide a non-trivial linear relation amongst the 3-jets of $\phi$ once the 2-jets have been fixed. We now make a linear
change of coordinates to ensure $\tilde\alpha_{111}\ne0$ in the new coordinate system and derive Assertion~(1).

Note that $\rho_{a}^D(\partial_{x^1},\partial_{x^2})$ is a real analytic function on $\tilde{\mathcal{J}}_k$ for any $k\ge3$. 
Thus either it vanishes identically or it does not vanish on an open dense set. We compute
$$
{ \rho_{a}^D(\partial_{x^1},\partial_{x^2})}={ \frac12}\left(\frac{\phi_{22}{ \phi_{12}}-\phi_2\phi_{122}}{\phi_2^2}+\frac{\phi_1 \phi_{112}-\phi_{11} { \phi_{12}}}{\phi_1^2}\right)\,.
$$
If we take
$$\begin{array}{ll}
\phi(0)=1,&\phi_1(0)=\phi_2(0)=1,\\[0.05in]
\phi_{11}(0)=\phi_{12}(0)=\phi_{22}(0)=0,&\phi_{122}=-\phi_{112}=1,
\end{array}$$
then Equation~(\ref{E24}) is satisfied and $ \rho_{a}^D(\partial_{x^1},\partial_{x^2})(0)\ne0$. Assertion~(2) follows.

Similarly, either $d\mathcal{E}$ vanishes identically or $d\mathcal{E}$ is non-zero on an open dense subset of $\tilde{\mathcal{J}}_k$ for any
$k\ge5$. We take $\phi(x^1,x^2)=\gamma(x^1+x^2)$; Equation~(\ref{E24}) becomes:
\begin{equation}\label{E26}
\gamma^{(3)}(t)=-\frac{\gamma'(t)^2\, \gamma''(t)-\mu\,  \gamma (t)\, \gamma''(t)^2}{\mu\,   \gamma (t)\, \gamma'(t)}\,.
\end{equation}
This ODE can be solved with arbitrary initial conditions $\{\gamma(0), \gamma'(0), \gamma''(0)\}$.
We have $\phi_2(0)=\phi_1(0)$ so we obtain 4 of the 8 components of $\tilde{\mathcal{J}}_5$; the other
4 components can be obtained by considering $\gamma(x^1-x^2)$.
A direct calculation shows that this connection is recurrent. We impose the identity of Equation~(\ref{E26}) and compute: 
$$\begin{array}{ll}
\rho^D=\frac{\gamma^{\prime\prime}}{\mu\gamma}\left(\begin{array}{ll}0&1\\1&0\end{array}\right),&
D\rho^D=-\frac{(1+\mu)\gamma^\prime}{\mu\gamma}(dx^1+dx^2)\otimes\rho^D\,,\\[0.2in]
\mathcal{E}={ 4\frac{(1+\mu)^2(\gamma^\prime)^2}{\mu\gamma\gamma^{\prime\prime}}},&
\dot{\mathcal{E}}={ 4}{\frac{(\mu +1)^2 \left(\mu \, \gamma\,\gamma'\,\gamma''-(\mu -1) \,(\gamma')^3\right)}{\mu^2\, \gamma^2\,\gamma''}}\,.
\end{array}$$
Since $\mu\ne-1$, $\dot{\mathcal{E}}$  is non-zero for generic values of $\{\gamma(0),\gamma'(0),\gamma''(0)\}$. Assertion~(3) follows.
\end{proof}

\begin{remark}\rm
Note that $\gamma(t):=t^\mu$ solves Equation~(\ref{E26}). For this choice of the defining function,
$\Gamma_{ii}{}^i=(\mu-1)(x^1+x^2)^{-1}$ and one obtains the example in Definition~\ref{E25}.
\end{remark}

\section{Affine surfaces supporting a parallel nilpotent $(1,1)$-tensor field}\label{S9}

In this section, we give examples of anti-self-dual quasi-Einstein manifolds which do not fit into the classification of Theorem~\ref{th:isotropic}, thus emphasizing the role of the Walker orientation. We will examine a special family of affine surfaces $(\Sigma,D)$ which admit a parallel nilpotent $(1,1)$-tensor field $T$ ($DT=0$, $T^2=0$); we refer to \cite{CLGRGRVL} for details. 
We assume the system of local coordinates $(x^1,x^2)$ is chosen so that: 
$
T\partial_{x^1}=\partial_{x^2}$ and $T\partial_{x^2}=0$.
Then $D T=0$ if and only if we have the relations:
\begin{equation}\label{eq:Christ-T-parallel}
\Gamma_{12}{}^1=0,\qquad\Gamma_{12}{}^2=\Gamma_{11}{}^1,\qquad\Gamma_{22}{}^1=0,\qquad\Gamma_{22}{}^2=0\,.
\end{equation}
Since $\operatorname{ker}T=\operatorname{span}\{\partial_{x^2}\}$ is parallel,  $\partial_{x^2}$ is a geodesic vector field. We have
\[
\rho^D_s= ({ \partial_{x^2}}\Gamma_{11}{}^2-{ \partial_{x^1}}\Gamma_{11}{}^1)\, dx^1\circ dx^1  \,,
\qquad 
\rho_a^D= { \partial_{x^2}}\Gamma_{11}{}^1\, dx^1\wedge dx^2\,.
\]
We shall suppose that  $\partial_{x^2}\Gamma_{11}{}^1=0$ to ensure that $\rho^D$ is symmetric. In this situation the Ricci tensor is recurrent and of rank one.
Work of Wong \cite{wong} shows that the local coordinates $(x^1,x^2)$ can be further specialized so that the only nonzero Christoffel symbol is ${}\Gamma_{11}{}^{2}$ and hence 
\begin{eqnarray*}
&&\rho^D\!=\!\partial_{x^2}\Gamma_{11}{}^2\left(\!\begin{array}{ll}1&\!0\\0&\!0\end{array}\!\right)\!,\;\\
&& D\rho^D\!=\!\left(\partial_{x^1}(\log\partial_{x^2}\Gamma_{11}{}^2) d x^1
+ \partial_{x^2}(\log\partial_{x^2}\Gamma_{11}{}^2) d x^2\right)\otimes\rho^D\!.
\end{eqnarray*}
Thus, with this choice of local coordinates the recurrence $1$-form $\omega$ is given by 
\begin{equation}\label{eq:recurrence-ricci}
\omega=\partial_{x^1}(\log\partial_{x^2}\Gamma_{11}{}^2) d x^1
+ \partial_{x^2}(\log\partial_{x^2}\Gamma_{11}{}^2) d x^2.
\end{equation}
We adopt the notation of Equation~\eqref{eq:metrictensorofriemannextensions}.
The modified Riemannian extension $g_{D,0,T,T} = \iota T \circ \iota T +g_D $ is never self-dual, but it is anti-self-dual if $\omega$  satisfies $\omega(\ker T)=0$ (see \cite{CLGRGRVL} for details). These affine surfaces are strongly 
projectively flat. Although Riemannian extensions of projectively flat connections are locally conformally flat, deformed Riemannian extensions $g_{D,\Phi}$
or modified Riemannian extensions $g_{D,\Phi,T,S}$ are not for generic tensors $\Phi$, $T$ and $S$. Consequently, the following result will show that there
exist examples of anti-self-dual quasi-Einstein Walker metrics that are never self-dual and hence not covered by the classification of
Theorem~\ref{th:isotropic}.

\begin{theorem}
\label{th:anti-self-dual}
Let $(\Sigma,D)$ be an affine surface. Assume that $\rho^D$ is symmetric, that $\operatorname{Rank}\{\rho^D\}=1$, and that
$D\rho^D=\omega\otimes\rho^D$.
Let $T$ be a parallel nilpotent $(1,1)$-tensor field on $\Sigma$, let $\hat f\in\mathcal{C}^\infty(\Sigma)$, and let $f=\pi^\ast\hat f$.
Assume that $d\hat f(\operatorname{ker}T)=0$.
\begin{enumerate}
\item If $\hat f$ satisfies Equation~(\ref{eq:QEE-affine}), then $(T^*\Sigma,g_{D,0,T,T},f,\mu)$ is an isotropic quasi-Einstein manifold with $\lambda=0$.
\smallbreak\item If $\omega(\operatorname{ker}T)=0$, then there exist local coordinates $(x^1,x^2)$ on $\Sigma$
so that the only nonzero Christoffel symbol is given by $\Gamma_{11}{}^2=u(x^1)+x^2v(x^1)$. 
If $\hat f(x^1)$ satisfies  $\hat f''(x^1)-\mu {\hat f'(x^1)}^2+2v(x^1)=0$,
then $(T^*\Sigma,g_{D,0,T,T},f,\mu)$ is an anti-self-dual  quasi-Einstein manifold which is not locally conformally flat.
\end{enumerate}\end{theorem}

\begin{proof}
Choose local coordinates $(x^1,x^2)$ on $\Sigma$ so that $T\partial_{x^1}=\partial_{x^2}$, $T\partial_{x^2}=0$ and so that
the only nonzero Christoffel symbol is $\Gamma_{11}{}^2$. Let $(x^1,x^2, x_{1'},x_{2'})$ be the induced coordinates on $T^*\Sigma$.
Let $\hat{f}\in\mathcal{C}^\infty(\Sigma)$ satisfy $d\hat f(\operatorname{ker}T)=0$. Then:
$$
(\operatorname{Hes}_{f}+\rho-\mu d f\otimes d f)(\partial_{x^1},\partial_{x^1})
=(\operatorname{Hes}^D_{\hat f}+2\rho^D_s-\mu d\hat f\otimes d\hat f)(\partial_{x^1},\partial_{x^1}),
$$
the other components being identically zero. This proves Assertion~(1).

To prove Assertion~(2), we assume that $\omega(\operatorname{ker}T)=0$. 
By Equation~\ref{eq:recurrence-ricci}, we have $\omega(\operatorname{ker}T)=0$ if and only if $\partial_{x^2}^2\Gamma_{11}{}^2=0$.
Consequently, $\Gamma_{11}{}^2$ is a linear function of $x^2$, i.e. $\Gamma_{11}{}^2=u(x^1)+x^2v(x^1)$. Finally, a direct computation shows 
that the only nonzero term in Equation~\eqref{eq:QEE-affine} is
\[
(\operatorname{Hes}^D_{\hat f}+2\rho^D_s-\mu d\hat f\otimes d\hat f)(\partial_{x^1},\partial_{x^1})=\hat{f}''(x^1)-\mu\hat{f}'(x^1)^2+2v(x^1)\,.
\]
This shows that the solutions of the quasi-Einstein equation are determined by the ODE $\hat{f}''(x^1)-\mu\hat{f}'(x^1)^2+2v(x^1)=0$.
\end{proof}

\end{document}